\documentclass[12pt]{amsart}%,reqno
\usepackage{amssymb,amsmath,amsthm}
\usepackage{a4}
\usepackage{color}
\usepackage{enumitem}
\usepackage[colorlinks=true,
linkcolor=webgreen,
filecolor=webbrown,
citecolor=webgreen]{hyperref}
\definecolor{webgreen}{rgb}{0,0,1}%{1,0.0,.6}
\definecolor{recrown}{rgb}{1,.2,.6}
\usepackage{fullpage}
\usepackage{orcidlink}
\begin{document}
\newtheorem{theorem}{Theorem}
\newtheorem{corollary}[theorem]{Corollary}
\newtheorem{lemma}[theorem]{Lemma}
\theoremstyle{definition}
\newtheorem*{example}{\bf Example}
\theoremstyle{theorem}
\newtheorem{conjecture}[theorem]{Conjecture}
\newtheorem{thmx}{\bf Theorem}
\renewcommand{\thethmx}{\text{\Alph{thmx}}}% "letter-numbered" theorems
\newtheorem{lemmax}{\bf Lemma}
\renewcommand{\thelemmax}{\text{\Alph{lemmax}}}% "
%\leftmargin=.5in
%\rightmargin=0.5in
%\textwidth=6truein
%\textheight=11.6truein
\hoffset=-0cm
%\voffset=+-2cm
\theoremstyle{definition}
\newtheorem*{definition}{Definition}
\theoremstyle{remark}
\newtheorem*{remark}{\bf Remark}
\theoremstyle{remark}
\newtheorem*{remarks}{\bf Remarks}
%\title{\bf On a Lemma of Singh and Kumar}
\parindent=0cm
\title{\bf A generalization of Dumas irreducibility criterion}
\author{Jitender Singh$^{1,\dagger}$ {\large \orcidlink{0000-0003-3706-8239}}}
\address[1]{Department of Mathematics,
Guru Nanak Dev University, Amritsar-143005, India\newline %\linebreak
{\tt jitender.math@gndu.ac.in}}
\markright{}
\date{}
\footnotetext[2]{Corresponding author email(s): {\tt jitender.math@gndu.ac.in}

2020MSC: {Primary 12E05; 11C08}\\

}
\maketitle
\newcommand{\K}{\mathbb{K}}
\begin{abstract}
Using Newton polygons, a key factorization result for polynomials over discrete valuation domains is proved, which in particular yields new irreducibility criteria including a generalization of the classical irreducibility criterion of Dumas.
\end{abstract}
\section{Introduction}
The classical irreducibility criteria due to Sch\"{o}nemann \cite{S}, Eisenstein \cite{E}, Perron \cite{P}, and  Dumas \cite{D} have become paradigms in the study of testing irreducibility of polynomials having integer coefficients. The irreducibility criterion due to Perron \cite{P} is based on a condition on the coefficients of a given polynomial forcing its all but one zeros to lie in the interior of  the closed unit disk in the complex plane, which immediately implies the irreducibility. On the other hand, each of the irreducibility results due to   Sch\"{o}nemann \cite{S} and Eisenstein \cite{E} rests on divisibility properties of the coefficients of a polynomial with respect to a prime number. In 1906, Dumas \cite{D} proved the following irreducibility criterion, which is a far reaching generalization of the irreducibility results of Sch\"{o}nemann \cite{S} and Eisenstein \cite{E}.
\begin{thmx}[Dumas \cite{D}]\label{th:A}
Let $(R,v)$  be a discrete valuation domain. Let $f=a_0 + a_1x + \cdots + a_nx^n \in {R}[x]$ be a primitive polynomial such that
\begin{enumerate}[label=$(\roman*)$]
\item $v(a_n)=0$,
\item $\frac{v(a_0)}{n}<\frac{v(a_i)}{n-i}$ for each $i=1,\ldots,n-1$,
\item $\gcd(v(a_0),~n) = 1$.
\end{enumerate}
Then the polynomial $f$ is irreducible in $R[x]$.
\end{thmx}
The proof of Theorem \ref{th:A} uses a fundamental result (see Theorem \ref{th:B}) of Dumas on Newton polygons. Theorem \ref{th:A} has been extended to polynomials over any valued field in the papers \cite{Kh2017,Kh23}.

The condition $(i)$ in the statement of  Theorem \ref{th:A} requires that the valuation of the leading coefficient of the underlying polynomial $f$ is zero. If  instead the valuation of a coefficient other than the leading coefficient of $f$ is zero,  can one still have any information about factorization of $f?$ What can be said about factorization of $f$ if the lower bound in the hypothesis $(ii)$ of Theorem \ref{th:A} is  allowed to occur at an index $\ell>0?$ Of course the  hypothesis $(iii)$ may require a modification in connection with that in the hypothesis $(ii)$. As we shall see in this paper, the answer to these questions is in the affirmative with appealing modifications of the hypotheses $(i)$-$(iii)$ of Theorem \ref{th:A}.
The purpose of the present article is to devise such conditions yielding new irreducibility results.  In view of this,  we first prove  a factorization result for polynomials over discrete valuation domains, which in particular yields Theorem \ref{th:A}. Further, our factorization result yields a class of irreducible polynomials having integer coefficients and also having the property that all zeros of such polynomials have the standard archimedean absolute value sufficiently large. In doing so, we use the well known Newton polygon technique. Main results of this paper are stated and discussed in Sec.~\ref{sec:1} and their proofs are presented in Sec.~\ref{sec:2}.
\section{Main results}\label{sec:1}
The main results of this paper are certainly true for polynomials over more general domains, which  require other technical tools. However, in the present article, we will restrict our attention to the polynomials over discrete valuation domains.
In view of this, our first factorization result is the following mild generalization of Theorem \ref{th:A}.
\begin{theorem}\label{th:1}
Let $(R,v)$  be a discrete valuation domain. Let $f=a_0 + a_1x + \cdots + a_nx^n \in {R}[x]$ be a primitive polynomial such that there exist indices $j$ and $\ell$ with $1\leq \ell+1 \leq j\leq n$ for which the following hold.
\begin{enumerate}[label=$(\roman*)$]
\item $v(a_j)=0$,
\item $\frac{v(a_\ell)}{j-\ell}<\frac{v(a_i)}{j-i}$ for each $i=0,1,\ldots, j-1$ with $i\neq \ell$,
\item $\gcd(v(a_\ell),~j-\ell) = 1$.
\end{enumerate}
Then any factorization $f(x)=f_1(x)f_2(x)$ of $f$ in $R[x]$ has a factor of degree $\geq j-\ell$. In particular, if $j=n$ and $\ell=0$, then the polynomial $f$ is irreducible in $R[x]$.
\end{theorem}
Note that Dumas irreducibility criterion (Theorem \ref{th:A})  is precisely the case $j=n$ and $\ell=0$ of Theorem \ref{th:1}. Further, the case $j=n$ of Theorem \ref{th:1} yields a generalization of the main result proved in \cite{W}.
\begin{remark}
The lower bound on the degree of a factor of $f$ as mentioned in Theorem \ref{th:1} is best possible in the sense that there exist polynomials attaining the said bound. For example, given a field $\K$, the degree valuation $v_\infty$ on the field of fractions $\K(x)$ of the polynomial ring $\K[x]$ is defined by setting $v_\infty(0)=\infty$, and for any $a,b\in \K[x]\setminus\{0\}$, $v_\infty(a/b)=\deg b-\deg a$. Then the valuation ring of $v_\infty$ in $\K(x)$ is the localization $(\K[1/x])_{(1/x)}$ of $\K[1/x]$ by the prime ideal $(1/x)$ of $\K[1/x]$. Consequently, $(\K[1/x]_{(1/x)},v_\infty)$ becomes a discrete valuation domain with respect to the valuation $v_\infty$ with the uniformizing parameter $1/x$ for the ring $\K[1/x]$.  Now consider the bivariate polynomial
\begin{eqnarray*}
f(1/x,y)&=&x^{-\ell}(1+y^{j-\ell})+x^{-j+\ell+1} y^\ell + y^j+y^{j+\ell}\\&&+(1-x^{-j+\ell+1})y^{2j-\ell}+y^{2j}\in \K[1/x]_{(1/x)}[y],
\end{eqnarray*}
which satisfies the hypothesis of Theorem \ref{th:1} for $1\leq \ell<j<2j=n$, $\ell\neq j-\ell$,  $a_0(x)=x^\ell=a_{j-\ell}(x)$,  $a_\ell(x)=x^{j-\ell-1}\in \K[x]$,  $a_j=1$, and $a_i=0$ for $i\not\in \{0,\ell,j-\ell,j, 2j-\ell,j+\ell,2j\}$. Here, we have $v_\infty(a_0(1/x))=\ell$, $v_\infty(a_\ell(1/x))=j-\ell-1$, and $v_\infty(a_i(1/x))=\infty$ for $i\in \{0,1,\ldots, j-1\}$ with $i\not\in\{0,\ell,j-\ell\}$. Consequently, we have $(i)$ $v_\infty(a_j(1/x))=0$,~$(ii)$ $\gcd(v_\infty(a_\ell(1/x)),j-\ell)=\gcd(j-\ell-1,j-\ell)=1$,  and
$(iii)$  $\frac{v_\infty(a_\ell(1/x))}{j-\ell}=1-\frac{1}{j-\ell}<1=\min_{i\neq \ell, i<j}\{\frac{v_\infty(a_i(1/x))}{j-i}\}$.  This in view of Theorem \ref{th:1} tells us that the polynomial $f$ has a factor of degree  $\geq (j-\ell)$. A direct computation tells us that $f$ factors as
\begin{eqnarray*}
f(1/x,y)&=&(1+y^{j-\ell})(x^{-\ell}+x^{-j+\ell+1}y^\ell+(1-x^{-j+\ell+1})y^{j}+y^{j+\ell}),
\end{eqnarray*}
with the factor  $1+y^{j-\ell}$ of degree $j-\ell$.
\end{remark}
As a consequence of Theorem \ref{th:1}, we may have the following factorization result.
\begin{corollary}\label{c:1}
Let $(R,v)$  be a discrete valuation domain. Let $f=a_0 + a_1x + \cdots + a_nx^n \in {R}[x]$ be a primitive polynomial such that there exist indices $j$ and $\ell$ with $1\leq \ell+1 \leq j\leq n$ for which the following hold.
\begin{enumerate}[label=$(\roman*)$]
\item $v(a_j)=0$,
\item $v(a_\ell)\leq \frac{{j-\ell}}{j} v(a_i)$ for each $i=0,1,\ldots, j-1$ with $i\neq \ell$,
\item $\gcd(v(a_\ell),~j-\ell) = 1$.
\end{enumerate}
Then any factorization $f(x)=f_1(x)f_2(x)$ of $f$ in $R[x]$ has a factor of degree $\geq j-\ell$. In particular, if $j=n$ and $\ell=0$, then the polynomial $f$ is irreducible in $R[x]$.
\end{corollary}
For a prime number $p$, let $v_p$ denote the $p$-adic valuation of the field of rational numbers with its valuation ring  $\mathbb{Z}$. Recall that $v_p(0)=\infty$, and for any nonzero integer $a$, $v_p(a)$ is the largest nonnegative integer such that $p^{v_p(a)}$ divides $a$.

Now we may have the following irreducibility criterion for a class of polynomials having integer coefficients.
\begin{theorem}\label{th:2}
Let $f=a_0 + a_1x + \cdots + a_nx^n \in \mathbb{Z}[x]$ be a primitive polynomial such that there exists a prime number $p$ and index $j$ with $1\leq j\leq n$ for which the following hold.
\begin{enumerate}[label=$(\roman*)$]
\item $v_p(a_j)=0$,
\item $\frac{v_p(a_0)}{j}<\frac{v_p(a_i)}{j-i}$ for each $i=1,\ldots,j-1$,
\item $\gcd(v_p(a_0),~j) = 1$.
\end{enumerate}
Let $d=a_0/p^{v_p(a_0)}$. If either $j=n$, or each zero $\theta$ of $f$ in ${\mathbb{C}}$ satisfies $|\theta|>d$, then  the polynomial $f$ is irreducible in $\mathbb{Z}[x]$.
\end{theorem}
\begin{remarks}\textbf{(1).}\label{r1} Let $f=a_0+a_1x+\cdots+a_n x^n\in \mathbb{Z}[x]$ be  such that
\begin{eqnarray*}
|a_0|>|a_1| d + \cdots + |a_{n-1}|d^{n-1}+ |a_n| d^n >0,
\end{eqnarray*}
for some positive real number $d$. If $x\in \mathbb{C}$ with $|x|\leq d$, then using the hypothesis we have
    \begin{eqnarray*}
|f(d)|\geq |a_0|-|a_1|d-|a_2|d^2-\cdots-|a_n|d^n>0.
    \end{eqnarray*}
This shows that each zero $\theta$ of $f$ must satisfy $|\theta|> d$.

\textbf{(2).}\label{r2} Let $f=a_0+a_1x+\cdots+a_n x^n\in \mathbb{Z}[x]$ be a polynomial satisfying conditions analogous to that of the well known Enestr\"om-Kakeya's theorem (see  \cite{Gardner} for a comprehensive review of Enestr\"om-Kakeya's theorem and its generalizations), that is, the coefficients of $f$ satisfy  
\begin{eqnarray*}
a_0\geq a_1\geq \cdots\geq a_{n-1}\geq a_n\geq 1.
\end{eqnarray*}
 We show that each zero of such a polynomial $f$ lies outside the open unit disk in the complex plane. So, assume on the contrary that there exists a zero $\theta \in \mathbb{C}$ of $f$ with $|\theta|<1$. We observe that
\begin{eqnarray*}
(x-1)f(x)=-a_0+\sum_{i=1}^n (a_{i-1}-a_i)x^i+a_n x^{n+1},
\end{eqnarray*}
which on taking $x=\theta$ with $f(\theta)=0$, we arrive at the following:
 \begin{eqnarray*}
 a_0=\sum_{i=1}^n (a_{i-1}-a_i)\theta^i+a_n \theta^{n+1}&\leq& \sum_{i=1}^n (a_{i-1}-a_i)|\theta|^i+a_n |\theta|^{n+1}\\&<&\sum_{i=1}^n (a_{i-1}-a_i)+a_n=a_0,
  \end{eqnarray*}
which is absurd. Thus, each zero $\theta\in \mathbb{C}$ of $f$ satisfies $|\theta|\geq 1$. If  such a polynomial $f$ has no cyclotomic factor, then each zero $\theta$ of $f$  satisfies $|\theta|>1$.

\textbf{(3).} If $f$ is a polynomial as in the preceding Remarks (1) or (2), and  the coefficients of $f$ satisfy the conditions $(i)$-$(iii)$ of Theorem \ref{th:2}, then it follows that such a polynomial $f$ is irreducible in $\mathbb{Z}[x]$.
\end{remarks}
 The next two irreducibility criteria may be deduced from Theorem \ref{th:2}.
\begin{corollary}\label{c:3}
Let $f=a_0 + a_1x + \cdots + a_nx^n \in \mathbb{Z}[x]$ be a primitive polynomial such that  there exists a prime number $p$ and index $j$ with $1\leq j\leq n$ for which the following hold.
\begin{enumerate}[label=$(\roman*)$]
\item $v_p(a_j)=0$,
\item $v_p(a_0)\leq v_p(a_i)$ for each $i=1,\ldots,j-1$,
\item $\gcd(v_p(a_0),~j) = 1$.
\end{enumerate}
Let $d=a_0/p^{v_p(a_0)}$. If either $j=n$, or each zero $\theta$ of $f$ in ${\mathbb{C}}$ satisfies $|\theta|>d$, then   $f$ is irreducible in $\mathbb{Z}[x]$.
\end{corollary}
We note that Corollary \ref{c:3} is precisely one of the main results proved in \cite{J-S-3}.
%\begin{example}
%For an odd prime number $p$  and positive integers $j$ and $n$  with $3\leq j<n$, we have
%\begin{eqnarray*}
%p^j\geq p^3=2p^2+p^2(p-2)> 2p^2+p+1.
%\end{eqnarray*}
%Consequently, all zeros of the polynomial
%\begin{eqnarray*}
%f=p^j+px+p^2 x^{j-1}+x^j+p^2x^n \in \mathbb{Z}[x]
%\end{eqnarray*}
%lie outside the closed unit disk in the complex plane. Further, $f$  satisfies rest of the hypothesis of Corollary \ref{c:3} for $\ell=1$. So, $f$ must be irreducible in $\mathbb{Z}[x]$.
%\end{example}
The following irreducibility criterion might be of independent interest.
\begin{corollary}\label{c:4}
Let $f=a_0 + a_1x + \cdots + a_nx^n \in \mathbb{Z}[x]$ be a primitive polynomial such that $a_0=\pm p^k d$ for a positive integer $d$ and a prime number $p\nmid d$. Let there be a  positive integer $m$ with $(km+1)\leq n$ for which
\begin{eqnarray*}
p^t\mid a_{(k-t) m+s},~p^{t+1}\nmid a_{((k-t) m+s},
\end{eqnarray*}
for every $s=1, \hdots, m$, and for every $t=1,\hdots,k$. Let $p$ does not divide $a_{km+1}$.
If either $km+1=n$, or each zero $\theta$ of $f$ in ${\mathbb{C}}$ satisfies $|\theta|>d$, then  $f$ is irreducible in $\mathbb{Z}[x]$.
\end{corollary}
\begin{example}
For a prime number $p$  and positive integers $k$ and $m$, the polynomial
\begin{eqnarray*}
f= p^k+px\Bigl(\frac{x^m-1}{x-1}\Bigr) \Bigl(\frac{x^{km}-p^k}{x^m-p}\Bigr)+x^{km+1}+x^{km+2}\in \mathbb{Z}[x]
%f= p^k+p^kx\frac{x^m-1}{x-1}(1+x^{m}/p+(x^{m}/p)^2+\cdots+(x^{m}/p)^{k-1})+x^{km+1}\in \mathbb{Z}[x]
\end{eqnarray*}
satisfies the hypothesis of Corollary \ref{c:4} for $n=km+2$, since we may express $f$ as
\begin{eqnarray*}
f&=&p^k+p^k\phi_{m}(x)+p^{k-1}x^m \phi_m(x)+\cdots+px^{(k-1)m}\phi_m(x)+x^{km+1}+x^{km+2},
%f= p^k+p^kx\frac{x^m-1}{x-1}(1+x^{m}/p+(x^{m}/p)^2+\cdots+(x^{m}/p)^{k-1})+x^{km+1}\in \mathbb{Z}[x]
\end{eqnarray*}
where we define
\begin{eqnarray*}
\phi_m(x)=x+x^2+\cdots+x^m,
\end{eqnarray*}
so that here, $a_{(k-t)m+s}=p^t$ for all $t=1,\ldots,k$ and $s=1,\ldots, m$. Since the coefficients of $f$ are monotonically decreasing positive integers starting from the constant term to the leading term and that $f$ has no cyclotomic factor, it follows  that all zeros of $f$ lie outside the closed unit disk in the complex plane. Thus, $f$ must be irreducible.
\end{example}
Finally, we may have the following factorization result.
\begin{theorem}\label{th:5}
Let $f=a_0+a_1 x +\cdots+a_nz^n\in \mathbb{Z}[x]$ be such that $a_0=\pm p_1^{k_1}\cdots p_r^{k_r}$, $r\geq 2$, where $p_1, \ldots, p_r$ are $r$ distinct primes and $k_1,\ldots,k_r$ are all positive integers. Suppose that for each $i=1,\ldots, r$, there exists  a positive integer $j_i\leq n$ for which the following hold.
\begin{enumerate}
\item[(i)] $v_{p_i}(a_{j_i})=0$,
\item[(ii)] $\frac{k_i}{j_i}<\frac{v_{p_i}(a_t)}{j_i-t}$ for each $t=1,\ldots,j_i-1$,
\item[(iii)] $\gcd(k_i,j_i)=1$.
\end{enumerate}
If each zero $\theta$ of $f$ in $\mathbb{C}$ satisfies $|\theta|>1$ for all $i=1,\ldots, r$, then  the polynomial $f$ is a product of at most $r$ irreducible factors in $\mathbb{Z}[x]$.
\end{theorem}
Note that factorization results analogous to Corollaries \ref{c:3} \& \ref{c:4} may be devised as particular cases of  Theorem \ref{th:5}.
\begin{remark}
Here also, the lower bound mentioned in Theorem \ref{th:5} is best possible, since the polynomial
\begin{eqnarray*}
f=x^r+\sigma_1 x^{r-1}+\cdots+\sigma_{r-1}x+\sigma_r=(x+{p_1}^{k_1})\cdots (x+{p_r}^{k_r})\in \mathbb{Z}[x],
\end{eqnarray*}
where $\sigma_i=\sum_{1\leq t_1<t_2<\cdots<t_{i}\leq n}p_{t_1}^{k_{t_1}}\cdots p_{t_i}^{k_{t_i}}$ for each $i=1,\ldots, r$ satisfies the hypothesis of Theorem \ref{th:5} for $n=r$ with  $j_1=j_2=\cdots=j_r=1$. Since $\sigma_i\leq \sigma_{i+1}$ for each $i=1,\ldots, r-1$, the coefficients of $f$ are strictly decreasing positive integers starting with the constant term, it follows that  $f$ has no cyclotomic factor and each zero $\theta\in \mathbb{C}$ of $f$ satisfies $|\theta|>1$. Thus, $f$ is a product of at most $r$ irreducible polynomials in $\mathbb{Z}[x]$, where we see that $f$ is actually a product  of exactly $r$ irreducible factors.
\end{remark}
\section{Proofs}\label{sec:2}
To prove our main results, we will recall some notations and definitions. In view of this, if $(R,v)$ is a discrete valuation domain, then for any polynomial $f=a_0+a_1x+\cdots+a_nx^n \in R[x]$ of degree $n\geq 1$, the Newton polygon $NP(f)$ of $f$ with respect to $v$ is defined as the lower convex hull of the points of the set
\begin{eqnarray*}
\{(i,v(a_i))~|~i=0,1,\ldots, n\}
\end{eqnarray*}
in the cartesian plane $\mathbb{R}^2$. By definition, $NP(f)$ is a polygonal path in $\mathbb{R}^2$ where each segment of this path is called an edge of $NP(f)$ and the point of intersection of two adjacent (non-collinear) edges on $NP(f)$ is called a vertex of $NP(f)$. Note that the sum of projections of all edges of $NP(f)$ on the $x$-axis in the cartesian plane $\mathbb{R}^2$ is equal to $\deg f$.

The following fundamental result of Dumas connects factorization of polynomials with the geometry of Newton polygons, and it will be used in our proofs.
\begin{thmx}[Dumas \cite{D}]\label{th:B}
Let $(R,v)$ be a discrete valuation domain. Let $f_1$ and $f_2$ be nonconstant polynomials in ${R}[x]$ with $f_1(0)f_2(0)\neq 0$. If $f(x)=f_1(x)f_2(x)$, then the edges in the Newton polygon  of $f$ with respect to $v$ may be formed by constructing a polygonal path composed by translates of all the edges that appear
in the Newton polygons of $f_1$ and $f_2$ with respect to $v$, using
exactly one translate for each edge, in such a way as to form a
polygonal path with increasing slopes.
\end{thmx}
The following well known result will also be used in the sequel.
\begin{lemma}\label{L:1}
Let $(R,v)$ be a discrete valuation domain. Let $f\in R[x]$ be a nonconstant polynomial. Corresponding to a segment $PQ$ joining the lattice points $P(a_1,v(a_1))$ and $Q(a_2,v(a_2))$ in the cartesian plane $\mathbb{R}^2$, the number of lattice points on the segment $PQ$ is equal to
\begin{eqnarray*}
1+\gcd(|a_1-a_2|, |v(a_1)-v(a_2)|).
\end{eqnarray*}
\end{lemma}
\begin{proof}[\bf Proof of Lemma \ref{L:1}] If we let $d=\gcd(|a_1-a_2|, |v(a_1)-v(a_2)|)$, then we may write
the  equation of the line joining the given points $P$ and $Q$ as
 \begin{eqnarray*}
 y=v(a_1)+\frac{v(a_1)-v(a_2)}{a_1-a_2}(x-a_1)=v(a_1)+\frac{(v(a_1)-v(a_2))}{d}\frac{(x-a_1)}{(a_1-a_2)/d}.
 \end{eqnarray*}
 Consequently, $(x,y)\in (\mathbb{Z}\times \mathbb{Z})\cap PQ$ if and only if $\frac{(x-a_1)}{(a_1-a_2)/d}\in \mathbb{Z}$, and that $a_1\leq x\leq a_2$. In view of this, we find that all lattice points $(x,y)$ lying on $PQ$ are explicitly given by
 \begin{eqnarray*}
 x=a_1+i\frac{(a_2-a_1)}{d},~y=v(a_1)+i\frac{v(a_2)-v(a_1)}{d},~i=0,1,\ldots,d,
  \end{eqnarray*}
  which are exactly $1+d$ in number.
\end{proof}
In view of Lemma \ref{L:1}, we note that the segment $PQ$ of $NP(f)$ has no lattice point on it other than the vertex points $P$ and $Q$ if and only if $\gcd(|a_1-a_2|, |v(a_1)-v(a_2)|)=1$.
\begin{proof}[\bf Proof of Theorem \ref{th:1}] By the hypotheses $(i)$ and $(ii)$ of the theorem, $NP(f)$ has an edge of negative slope joining the vertices $P(\ell,v(a_\ell))$ and $Q(j,0)$. In view of the hypothesis $(iii)$, we have from Lemma \ref{L:1} that the segment $PQ$ has no lattice point on it other than the end points $P$ and $Q$. By Theorem \ref{th:B}, the segment $PQ$ on $NP(f)$ is a translate of unique edge of Newton polygon of one of $f_1$ and $f_2$. Suppose that $PQ$ is a translate of an edge of $NP(f_1)$. Then $\deg f_1$ must be at least the length of the horizontal projection of $PQ$, which is $j-\ell$. 
\end{proof}
\begin{proof}[\bf Proof of Corollary \ref{c:1}]
The given polynomial satisfies all the hypothesis of Theorem \ref{th:1}. To see this, we only need to verify the hypothesis $(ii)$ of Theorem \ref{th:1} for $f$. We proceed as follows. Since $v(a_\ell)\leq \frac{j-\ell}{j}v(a_i)$ for each $i=0,1,\ldots, j-1$ with $i\neq \ell$, we have
\begin{eqnarray*}
v(a_\ell)\leq \frac{j-\ell}{j}v(a_i)<\frac{j-\ell}{j}\times \frac{j}{j-i}v(a_i)=\frac{j-\ell}{j-i}v(a_i),~i=0,1,\ldots, j-1,~i\neq \ell,
\end{eqnarray*}
which is precisely the hypothesis $(ii)$ of Theorem \ref{th:1}. So, any factorization $f(x)=f_1(x)f_2(x)$ of $f$ in $R[x]$ has a factor of degree $\geq j-\ell$.
\end{proof}
Now we may have the following factorization result, which for the case $\ell=0$ will be used in proving Theorems \ref{th:2} and \ref{th:5} and it might be of independent interest as well.
\begin{lemma}\label{L:2}
Let $(R,v)$  be a discrete valuation domain. Let $f=a_0 + a_1x + \cdots + a_nx^n \in {R}[x]$ be such that there exist indices  $j$ and $\ell$ with $1\leq \ell+1\leq j\leq n$ for which the following hold.
\begin{enumerate}[label=$(\roman*)$]
\item $v(a_j)=0$,
\item $\frac{v(a_\ell)}{j-\ell}<\frac{v(a_i)}{j-i}$ for each $i=0,\ldots, j-1$ with $i\neq \ell$,
\item $\gcd(v(a_\ell),~j-\ell) = 1$,
\item[]If $\ell>1$, then each of the following two conditions also holds.
\item  $\frac{v(a_0)-v(a_\ell)}{\ell}>\frac{v(a_0)-v(a_i)}{i}$ for each $i=1,\ldots, \ell-1$,
\item  $\gcd(v(a_0)-v(a_\ell),\ell)=1$.
\end{enumerate}
Then for any factorization $f(x)=f_1(x)f_2(x)$ of $f$ in $R[x]$ one of $v(f_1(0))$ and $v(f_2(0))$ belongs to the set $\{0,~v(a_\ell)\}$.
\end{lemma}
\begin{proof}[\bf Proof of Lemma \ref{L:2}] Since $v$ is a discrete valuation on $R$, $v$ is nonnegative on $R$. By the given hypotheses $(i)$ and $(ii)$ of the lemma, we find that in the Newton polygon  $NP(f)$ of $f$ with respect to $v$, there exists an edge joining the vertices $P(\ell,v(a_\ell))$ and $Q(j,v(a_j))=(j,0)$ of the negative slope $-v(a_\ell)/(j-\ell)$. Since $\gcd(v(a_\ell),~j-\ell)=1$, it follows from Lemma \ref{L:1} that there is no lattice point on the edge $PQ$ other than its  end points $P$ and $Q$. Further, we note that $PQ$ is the rightmost edge of $NP(f)$ having negative slope. 
Let
\begin{eqnarray*}
f_1=b_0+b_1 x+\cdots+b_m x^m,~f_2=c_0+c_1 x+\cdots+c_{n-m} x^{n-m},
\end{eqnarray*}
in $R[x]$ with $\min\{m,n-m\}\geq 1$, where we also define $b_{m+i}=0=c_{n-m+i}$ for all $i\geq 1$.  Then we have $f_1(0)=b_0$ and $f_2(0)=c_0$ and $a_0=b_0c_0$ so that $v(a_0)=v(b_0)+v(c_0)$.  Since $v(a_j)=0$, there exist smallest indices $s\leq m$ and $t\leq n-m$ with $s+t=j$ for which $v(b_s)=0=v(c_t)$. 

We first assume that $\ell>1$. Suppose that none of  $v(f_1(0))$ and $v(f_2(0))$ is zero. Then $s>0$ and $t>0$. Then in view of the hypotheses $(ii)$, $(iv)$, and $(v)$ and Lemma \ref{L:1}, we find that there is another edge joining the vertices $M(0,v(a_0))$ and $P(\ell,v(a_\ell))$ on $NP(f)$ having no lattice point on $MP$ other than $M$ and $P$ and of negative slope $-(v(a_0)-v(a_\ell))/\ell$ not equal to $-v(a_\ell)/(j-\ell)$. We conclude that $NP(f)$ has  exactly two edges of negative slopes and are non-collinear. There exist least indices  $\alpha$ and $\beta$ for which we have
\begin{eqnarray*}
 \frac{v(b_\alpha)}{s-\alpha}&=&\min_{0\leq i\leq s-1}\Bigl\{\frac{v(b_i)}{s-i}\Bigr\},~
 \frac{v(c_\beta)}{t-\beta}=\min_{0\leq i\leq t-1}\Bigl\{\frac{v(c_i)}{t-i}\Bigr\}.
 \end{eqnarray*}
Observe that $v(b_\alpha)>0$ and $v(c_\beta)>0$, since $\alpha<s$ and $\beta<t$.
Thus, there is an edge on  $NP(f_1)$  joining the vertices $M_1(\alpha,v(b_\alpha))$ and $Q_1(s,0)$ of negative slope $-v(b_\alpha)/(s-\alpha)$, and there is an edge on $NP(f_2)$ joining the vertices $M_2(\beta,v(c_\beta))$ and $Q_2(t,0)$ of negative slope $-v(c_\beta)/(t-\beta)$. This in view of the fact that $NP(f)$ has exactly two edges of negative slopes and are non-collinear along with  the use of Theorem \ref{th:B} tells us that the edges $MP$ and $PQ$ of $NP(f)$ of negative slopes must be translates of the edges $M_1Q_1$ of $NP(f_1)$ and $M_2Q_2$ of $NP(f_2)$. Assume that $MP$ and $M_1Q_1$ are parallel, and $PQ$ and $M_2Q_2$ are parallel. Then we have
\begin{eqnarray*}
 \frac{v(a_0)-v(a_\ell)}{\ell}=\frac{v(b_\alpha)}{s-\alpha},~\frac{v(a_\ell)}{j-\ell}=\frac{v(c_\beta)}{t-\beta},
 \end{eqnarray*}
so that $(0,v(a_0))=(\alpha+u,v(b_\alpha)+v(c_\beta)),~(\beta+s,v(c_\beta))=(s+u,v(c_\beta))=(\ell,v(a_\ell))$ for some integer $u$. From these equations we have $\alpha+u=0$, $\beta+s=s+u=\ell$ and so, $u=\beta$ and $\ell=s+\beta$. Thus, $\alpha+\beta=0$ which gives $\alpha=0=\beta$.  Since $s=s+u=\ell$  and $v(c_0)=v(c_\beta)=v(a_\ell)$, we have $t=j-\ell$ and so, we arrive at the following.
\begin{eqnarray*}
 \frac{v(a_0)-v(a_\ell)}{\ell}=\frac{v(b_0)}{s}=\frac{v(b_0)}{\ell},~\frac{v(a_\ell)}{j-\ell}=\frac{v(c_0)}{t}=
\frac{v(c_0)}{j-\ell},
\end{eqnarray*}
from which we get  $v(a_0)-v(a_\ell)=v(b_0)$ and $v(a_\ell)=v(c_0)$. This proves the lemma for the case when $\ell>1$.

The proof in the case $\ell=1$ is similar to the case when $\ell>1$, and so, we omit the proof.

Now suppose that $\ell=0$. In this case, we observe from the given hypothesis that $NP(f)$ has only one edge of negative slope $-v(a_0)/j$ joining the vertices $A(0,v(a_0))$ and $B(j,0)$ with no lattice point on $AB$ other than the end points $A$ and $B$ and with horizontal projection $j$. This in view of Theorem \ref{th:B} tells us that exactly one of $NP(f_1)$ and $NP(f_2)$ has unique edge of negative slope whose translate is the edge $AB$ of $NP(f)$. Assume that $AB$ is translate of the edge of $NP(f_1)$ having negative slope $-v(a_0)/j$ and horizontal projection $j$ in the cartesian plane $\mathbb{R}^2$. This forces $s=j$ and $t=0$, so that in this case we have  $v(a_0)=v(b_0)$ and $v(c_0)=0$. 
\end{proof}
The case $\ell=0$ of Lemma \ref{L:2} is also interesting and we record it separately here for its later use.
\begin{lemma}\label{L:3}
Let $(R,v)$  be a discrete valuation domain. Let $f=a_0 + a_1x + \cdots + a_nx^n \in {R}[x]$ be such that there exists an index  $j$  with $1\leq j\leq n$ for which the following hold.
\begin{enumerate}[label=$(\roman*)$]
\item $v(a_j)=0$,
\item $\frac{v(a_0)}{j}<\frac{v(a_i)}{j-i}$ for each $i=1,\ldots, j-1$,
\item $\gcd(v(a_0),~j) = 1$.
\end{enumerate}
Then any factorization $f(x)=f_1(x)f_2(x)$ of $f$ in $R[x]$ satisfies $v(f_1(0))=v(a_0)$, or $v(f_2(0))=v(a_0)$.
\end{lemma}
Note that Lemma \ref{L:3} in particular yields the fundamental results of the papers \cite[Lemma 1]{J-S-3} and  \cite[Lemma 1.4]{Bonciocat2013}.
\begin{proof}[\bf Proof of Theorem \ref{th:2}] The case $j=n$ follows from Theorem \ref{th:1} for $\ell=0$. Now assume that $j<n$. Suppose on the contrary that $f(x)=f_1(x)f_2(x)$ for nonconstant polynomials $f_1$ and $f_2$ in $\mathbb{Z}[x]$. Applying Lemma \ref{L:3} for $v=v_p$, the $p$-adic valuation of $\mathbb{Q}$, we have
 \begin{eqnarray}\label{e1}
v_p(f_1(0))=0,~\text{or}~v_p(f_2(0))=0,
 \end{eqnarray}
 which shows that $p$ does not divide at least one of $f_1(0)$ and $f_2(0)$. Assume that $p$ does not divide $|f_1(0)|$. If we define $v_p(a_0)=k$, then as $\pm p^kd=a_0=f(0)=f_1(0)f_2(0)$, it follows that $p^k$ divides $|f_2(0)|$, since $p\nmid d$. Consequently, we must have $|f_1(0)|\leq d$.

If $\theta_1,\ldots, \theta_{\deg f_1}$ are all zeros of $f_1$ in ${\mathbb{C}}$, then $|\theta_i|>d$ for all such $i$. If $0\neq \alpha \in \mathbb{Z}$ denotes the leading coefficient of $f_1$, then we may express $f_1$ as
 \begin{eqnarray*}
f_1(x)=\alpha(x-\theta_1)\cdots (x-\theta_{\deg f_1}).
\end{eqnarray*}
We then have
\begin{eqnarray*}
d\geq |f_1(0)|=|\alpha||\theta_1|\cdots |\theta_{\deg f_1}|>|\alpha_1| d^{\deg f_1}\geq d,
\end{eqnarray*}
which is a contradiction.
\end{proof}
\begin{proof}[\bf Proof of Corollary \ref{c:3}]
The given polynomial satisfies all the hypothesis of Theorem \ref{th:2}. To see this, we only need to verify the hypothesis $(ii)$ of Theorem \ref{th:2} for $f$. We proceed as follows. Since $v_p(a_0)\leq v_p(a_i)$ for each $i=1,\ldots, j-1$, we have \begin{eqnarray*}
v_p(a_0)\leq v_p(a_i)<\frac{j}{j-i}v_p(a_i),~i=1,\ldots, j-1,
\end{eqnarray*}
and so, the hypothesis $(ii)$ of Theorem \ref{th:2} holds. Thus, the polynomial $f$ is irreducible in $\mathbb{Z}[x]$.
\end{proof}
%In view of the calculations made in the proof of Corollary \ref{c:3}, we observe that Lemma \ref{L:2} yields the following result which might be of independent interest.
%\begin{lemma}\label{L:3}
%Let $(R,v)$  be a discrete valuation domain. Let $f=a_0 + a_1x + \cdots + a_nx^n \in {R}[x]$ be such that there exist indices  $j$ and $\ell$ with $1\leq \ell+1\leq j\leq n$ for which the following hold.
%\begin{enumerate}[label=$(\roman*)$]
%\item $v(a_j)=0$,
%\item ${v(a_\ell)}<\frac{j-\ell}{j}{v(a_i)}$ for each $i=1,\ldots, j-1$ with $i\neq \ell$,
%\item $\gcd(v(a_\ell),~j-\ell) = 1$.
%\end{enumerate}
%Then any factorization $f(x)=f_1(x)f_2(x)$ of $f$ in $R[x]$ satisfies $v(f_1(0))=0$, or $v(f_2(0))=0$.
%\end{lemma}
%Note that the case $\ell=0$ of Lemma \ref{L:3} is precisely \cite[Lemma 1]{J-S-3}.
\begin{proof}[\bf Proof of Corollary \ref{c:4}]
Taking $j=km+1$ in Theorem \ref{th:2} we  find that for all $s=1,\ldots, m$ and $t=1,\ldots, k$, we have
\begin{eqnarray*}
k\{km+1-(k-t)m-s\}-t(km+1)=-k(s-1)-t\leq -t<0.
\end{eqnarray*}
Consequently, we have the following for all $s=1,\ldots m,~t=1,\ldots, k$.
\begin{eqnarray*}
\frac{v_p(a_0)}{j}=\frac{k}{km+1}<\frac{t}{km+1-(k-t)m-i}=\frac{v_p(a_{(k-t)m+s})}{j-((k-t)m+s)}.
\end{eqnarray*}
So, by Theorem \ref{th:2}, the polynomial $f$ is irreducible in $\mathbb{Z}[x]$.
\end{proof}
\begin{proof}[\bf Proof of Theorem \ref{th:5}]Let $f(x)=f_1(x)f_2(x)\cdots f_s(x)$, where each of $f_1,f_2,\ldots, f_s$ is an irreducible polynomial in $\mathbb{Z}[x]$, and  $s\geq 2$ is a positive integer. Assume on the contrary that $s>r$. Since each zero of $f$ lies outside the closed unit disk in the complex plane, it follows that $|f_t(0)|> 1$ for each $t=1,\ldots, s$. This along with the fact that
\begin{eqnarray*}
\pm p_1^{k_1}\cdots p_r^{k_r}=f(0)=f_1(0)f_2(0)\cdots f_s(0),
\end{eqnarray*}
and that $s>r$, tells us that there exist at least two distinct indices $\ell, m \in \{1,\ldots, s\}$  and an index $i\in \{1,\ldots, r\}$ for which $p_i$ divides each of $f_\ell(0)$ and $f_{m}(0)$, that is, $v_{p_i}(f_\ell(0))>0$ and $v_{p_i}(f_m(0))>0$. If we choose a partition of the set of indices $\{1,2,\ldots, s\}$ into two disjoint subsets $A$ and $B$ for which $\ell$ belongs to $A$ and $m$ belongs to $B$, then
we have the factorization
\begin{eqnarray*}
f(x)=\prod_{t\in A} f_t(x) \prod_{t'\in B} f_{t'}(x),
\end{eqnarray*}
into the polynomials $\prod_{t\in A} f_t(x)$ and $\prod_{t'\in B} f_{t'}(x)$ in $\mathbb{Z}[x]$ for which we have
\begin{eqnarray*}
v_{p_i}\Bigl(\prod_{t\in A} f_t(0)\Bigr)\geq v_{p_i}(f_\ell(0))>0;~ v_{p_i}\Bigl(\prod_{t'\in B} f_{t'}(0)\Bigr)\geq v_{p_i}(f_m(0))>0.
\end{eqnarray*}
This contradicts Lemma \ref{L:3} for $p=p_i$, and so, we must have $s\leq r$.
\end{proof}
\subsection*{Disclosure statement}
The author reports to have no competing interests to declare.

\end{document}